[1994/12/01]
\documentclass{ijmart-mod}
\chardef\bslash=`\\ 

\usepackage{graphicx}
\usepackage[breaklinks=true]{hyperref}
\usepackage{mathtools}
\usepackage{tikz-cd}

\newtheorem{thm}{Theorem}[section]
\newtheorem{cor}[thm]{Corollary}
\newtheorem{lem}[thm]{Lemma}
\newtheorem{prop}[thm]{Proposition}

\theoremstyle{definition}
\newtheorem{defn}[thm]{Definition}
\newtheorem{rem}[thm]{Remark}
\newtheorem{qtn}[thm]{Question}

\theoremstyle{remark}

\newcommand{\eval}[2][\right]{\relax
  \ifx#1\right\relax \left.\fi#2#1\rvert}

\begin{document}


\title{Many equiprojective polytopes}

\author[T. Buffi{\`e}re]{Th{\'e}ophile Buffi{\`e}re}
\address{Universit{\'e} Paris 13, Villetaneuse, France}
\email{buffiere@math.univ-paris13.fr}

\author[L. Pournin]{Lionel Pournin}
\address{Universit{\'e} Paris 13, Villetaneuse, France}
\email{lionel.pournin@univ-paris13.fr}

\begin{abstract}
A $3$-dimensional polytope $P$ is $k$-equiprojective when the projection of $P$ along any line that is not parallel to a facet of $P$ is a polygon with $k$ vertices. In 1968, Geoffrey Shephard asked for a description of all equiprojective polytopes. It has been shown recently that the number of combinatorial types of $k$-equiprojective polytopes is at least linear as a function of $k$. Here, it is shown that there are at least $k^{3k/2+o(k)}$ such combinatorial types as $k$ goes to infinity. This relies on the Goodman--Pollack lower bound on the number of order types and on new constructions of equiprojective polytopes via Minkowski sums.
\end{abstract}
\maketitle

\section{Introduction}\label{BP.sec.0}

In 1968, Geoffrey Shephard asked a number of questions related to the combinatorics of Euclidean polytopes---convex hulls of finitely many points from $\mathbb{R}^d$ \cite{Shephard1968a,Shephard1968b}. Among them, Question IX asks for a method to construct every \emph{equiprojective polytope}. A $3$-dimensional polytope $P$ is $k$-equiprojective when its orthogonal projection on any plane (except for the planes that are orthogonal to a facet of $P$) is a polygon with $k$ vertices. A straightforward example of equiprojective polytopes is provided by prisms: a prism over a polygon with $k-2$ vertices is $k$-equiprojective. Hallard Croft, Kenneth Falconer, and Richard Guy recall Shephard's question in their book about unsolved problems of geometry \cite[Problem B10]{CroftFalconerGuy1991}. While a practical criterion (see Theorem \ref{BP.sec.4.thm.1} to follow) and an algorithm to recognise equiprojective polytopes have been proposed by Masud Hasan and Anna Lubiw \cite{HasanLubiw2008}, the problem is still open.

Some equiprojective polytopes have been recently constructed by truncating Johnson solids ($3$-dimensional polytopes whose facets are regular polygons) and by gluing two well-chosen prisms along a facet \cite{HasanHossainLopez-OrtizNusratQuaderRahman2022}. The latter construction shows, in particular, that the number of different combinatorial types of $k$\nobreakdash-equiprojective polytopes is at least a linear function of $k$ (recall that two polytopes have the same combinatorial type when their face lattices are isomorphic). Here, this result is improved as follows.
\begin{thm}\label{BP.sec.0.thm.1}
There are at least
$$
k^{k\left(\frac{3}{2}+O\left(\frac{1}{\mathrm{log}\,k}\right)\right)}
$$
different combinatorial types of $k$-equiprojective polytopes.
\end{thm}

Our proof of Theorem \ref{BP.sec.0.thm.1} is split into two cases depending on the parity of~$k$. Recall that a zonotope $Z$ is a Minkowski sum of line segments. It can be assumed without loss of generality that these segments are pairwise non-collinear, in which case their number is the \emph{number of generators of $Z$}. As observed in~\cite{HasanLubiw2008}, a $3$\nobreakdash-dimensional zonotope with $n$ generators is a $2n$-equiprojective polytope. When $k$ is even, this observation is the first ingredient in the proof of Theorem \ref{BP.sec.0.thm.1}. The second ingredient is the following estimate for the number of combinatorial types of zonotopes that may be of independent interest.
\begin{thm}\label{BP.sec.1.thm.2}
The number $z(n,d)$ of combinatorial types of $d$-dimensional zonotopes with $n$ generators satisfies
$$
n^{(d^2-2d)n\left(1+O\left(\frac{1}{\mathrm{log}\,n}\right)\right)}\leq{z(n,d)}\leq{n^{(d-1)^2n\left(1+O\left(\frac{\mathrm{log}\,\mathrm{log}\,n}{\mathrm{log}\,n}\right)\right)}}
$$
when $d$ is fixed and $n$ goes to infinity.
\end{thm}

We establish Theorem \ref{BP.sec.1.thm.2} as a consequence of the Goodman--Pollack bound the number of order types \cite{GoodmanPollack1986} and its refinement by Noga Alon \cite{Alon1986}.

When $k$ is odd, Theorem \ref{BP.sec.0.thm.1} requires a different construction that uses a Minkowski sum with a triangle. In order to analyse how equiprojectivity behaves under Minkowski sums, we rely on the notion of an \emph{aggregated cone} of a polytope at one of its edge directions. Roughly, an edge direction of a polytope $P$ contained in $\mathbb{R}^3$ is a vector $u$ in the unit sphere $\mathbb{S}^2$ parallel to an edge of $P$. The aggregated cone $C_P(u)$ of $P$ at $u$ is the union of the $2$-dimensional normal cones of $P$ contained in the plane $u^\perp$ through the origin of $\mathbb{R}^3$ and orthogonal to $u$. We obtain the following characterization of equiprojectivity.

\begin{thm}\label{BP.sec.0.thm.2}
A $3$-dimensional polytope $P$ is equiprojective if and only if, for every edge direction $u$ of $P$, either
\begin{itemize}
\item[(i)] the aggregated cone $C_P(u)$ is equal to $u^\perp$ or
\item[(ii)] the relative interior of $-C_P(u)$ is equal to $u^\perp\mathord{\setminus}C_P(u)$.
\end{itemize}
\end{thm}

Since in practice, computing the faces of a Minkowski sum of polytopes is done via their normal fans, Theorem \ref{BP.sec.0.thm.2} provides a way to prove Theorem \ref{BP.sec.0.thm.1} in the case when $k$ is odd. More generally, Theorem \ref{BP.sec.0.thm.2} allows to construct new classes of equiprojective polytopes. For instance, we prove the following.

\begin{thm}\label{BP.sec.0.thm.3}
Consider a $3$-dimensional polytope $P$ obtained as a Minkowski sum of finitely many polygons. If no two of these polygons share an edge direction, then $P$ is an equiprojective polytope.
\end{thm}

More generally, we will provide a condition under which a Minkowski sum of equiprojective polytopes, polygons, and line segments (that are allowed to share edge directions) is equiprojective (see Theorem \ref{BP.sec.4.thm.5}). We will also explain how the value of $k$ such that this Minkowski sum is $k$-equiprojective can be computed from the aggregated cones of the summands (see Theorem \ref{BP.sec.4.thm.6}). 

The article is organized as follows. We prove Theorem \ref{BP.sec.1.thm.2} and derive from it the special case of Theorem \ref{BP.sec.0.thm.1} when $k$ is even in Section~\ref{BP.sec.1}. We introduce the aggregated cones of a polytope and establish Theorem~\ref{BP.sec.0.thm.2} in Section \ref{BP.sec.2}. We give the announced Minkowski sum constructions of equiprojective polytopes and prove Theorem~\ref{BP.sec.0.thm.3} in Section \ref{BP.sec.2.5}. We provide the proof of Theorem \ref{BP.sec.0.thm.1} in the case when $k$ is odd in Section \ref{BP.sec.3}. Finally, we conclude the article with Section \ref{BP.sec.5}, where some questions in the spirit of Shephard's are stated about the decomposability of equiprojective polytopes.

\section{The combinatorial types of zonotopes}\label{BP.sec.1}

Up to translation, a zonotope is any subset of $\mathbb{R}^d$ of the form
$$
Z=\sum_{g\in\mathcal{G}}\mathrm{conv}\{0,g\}
$$
where $\mathcal{G}$ is a finite, non-empty set of pairwise non-collinear vectors from $\mathbb{R}^d\mathord{\setminus}\{0\}$, which we refer to as a \emph{set of generators} of $Z$. Note that $Z$ admits several sets of generators, each obtained from $\mathcal{G}$ by negating a part of the vectors it contains. In particular, $Z$ has $2^{|\mathcal{G}|}$ sets of generators, each of the same cardinality. We will refer to this common cardinality as the \emph{number of generators} of $Z$.

Recall that the face lattice of a polytope is the set of its faces ordered by inclusion and that two polytopes have the same combinatorial type if their face lattices are isomorphic (here, by an isomorphism, we mean a bijection that preserves face inclusion). The goal of this section is to estimate the number of combinatorial types of zonotopes in terms of their number of generators. This will allow to prove Theorem \ref{BP.sec.0.thm.1} when $k$ is even thanks to the following statement from \cite{HasanLubiw2008}, which we provide an alternative proof for.

\begin{prop}\label{BP.sec.1.prop.1}
A $3$-dimensional zonotope $Z$ is a $k$-equiprojective polytope, where $k$ is twice the number of generators of $Z$.
\end{prop}
\begin{proof}
Consider a $3$-dimensional zonotope $Z$ contained in $\mathbb{R}^3$ and denote by $\mathcal{G}$ a set of generators of $Z$. Further consider a plane $H$, also contained in $\mathbb{R}^3$, that is not orthogonal to a facet of $Z$ and denote by $\pi:\mathbb{R}^3\rightarrow{H}$ the orthogonal projection on $H$. By construction $\pi(Z)$ can be expressed as
$$
\pi(Z)=\sum_{g\in\pi(\mathcal{G})}\mathrm{conv}\{0,g\}
$$
up to translation. As an immediate consequence, $\pi(Z)$ is a Minkowski sum of line segments contained in $H$ and, therefore, it is a $2$-dimensional zonotope (or in other words, a zonogon). Since $H$ is not orthogonal to a facet of $Z$, the orthogonal projections on $H$ of two distinct generators of $Z$ cannot be collinear. Hence, $Z$ and $\pi(Z)$ have the same number of generators.

Finally, recall that the number of vertices of a zonogon is twice the number of its generators. Therefore, we have shown that the number of vertices of the orthogonal projection of $Z$ on a plane that is not orthogonal to any of its facets is always twice the number of generators of $Z$, as desired.
\end{proof}

It is well known that the combinatorial types of zonotopes are determined by the \emph{oriented matroids} of their sets of generators \cite{BjornerLasVergnasSturmfelsWhiteZiegler1999}. Let us recall what the oriented matroids of a finite subset $\mathcal{X}$ of $\mathbb{R}^d\mathord{\setminus}\{0\}$ are. First pick a bijection
$$
\sigma:\{1,\ldots,|\mathcal{X}|\}\rightarrow\mathcal{X}
$$
whose role is to order the vectors of $\mathcal{X}$ so that each of them corresponds to a coordinate of $\mathbb{R}^{|\mathcal{X}|}$. A vector $z$ from $\{-1,0,1\}^{|\mathcal{X}|}$ is a \emph{covector of $\mathcal{X}$} with respect to $\sigma$ when there exists a non-zero vector $y$ in $\mathbb{R}^d$ such that $z_i$ is the sign of $\sigma(i)\mathord{\cdot}y$ for every $i$. The set $M_\sigma(\mathcal{X})$ of all the covectors of $\mathcal{X}$ with respect to $\sigma$ is an \emph{oriented matroid of $\mathcal{X}$}. Note that the other oriented matroids of $\mathcal{X}$ can be obtained by varying $\sigma$ or, equivalently, by letting the isometries of $\mathbb{R}^{|\mathcal{X}|}$ that permute the coordinates act on $M_\sigma(\mathcal{X})$. For more details on oriented matroids, see for instance \cite{BjornerLasVergnasSturmfelsWhiteZiegler1999} or \cite{Richter-GebertZiegler2017}. Two finite subsets $\mathcal{X}$ and $\mathcal{X}'$ of $\mathbb{R}^d\mathord{\setminus}\{0\}$ are called \emph{oriented matroid equivalent} when they have at least one oriented matroid in common. It is easy to check that oriented matroid equivalence is indeed an equivalence relation on the finite subsets of $\mathbb{R}^d\mathord{\setminus}\{0\}$.

The following statement is proven in \cite{BjornerLasVergnasSturmfelsWhiteZiegler1999} (see Corollary 2.2.3 therein). It provides the announced correspondence between the combinatorial type of a zonotope and the oriented matroids of its sets of generators.

\begin{prop}\label{BP.sec.1.prop.2}
Two zonotopes $Z$ and $Z'$ have the same combinatorial type if and only if for every set $\mathcal{G}$ of generators of $Z$, there exists a set $\mathcal{G}'$ of generators of $Z'$ such that $\mathcal{G}$ and $\mathcal{G}'$ are oriented matroid equivalent.
\end{prop}

According to Proposition \ref{BP.sec.1.prop.2}, counting the number of combinatorial types of $d$-dimensional zonotopes with $n$ generators amounts to counting the number of oriented matroids of $d$-dimensional sets of $n$ vectors from $\mathbb{R}^d\mathord{\setminus}\{0\}$. Estimates on these numbers have been given by Jacob Goodman and Richard Pollack \cite{GoodmanPollack1986} and  by Noga Alon \cite{Alon1986} in the terminology of order types \cite{Eppstein2018,FelsnerGoodman2017,GoaocWelzl2023,GoodmanPollack1983,Matousek2002}. Theorem~4.1 from~\cite{Alon1986} 
can be rephrased as follows in terms of oriented matroids. Observe that the lower bound in that statement is established in \cite[Section 5]{GoodmanPollack1986}.

\begin{thm}\label{BP.sec.1.thm.1}
The number $t(n,d)$ of oriented matroids of sets of $n$ vectors that span $\mathbb{R}^d$ and whose last coordinate is positive satisfies
$$
\biggl(\frac{n}{d-1}\biggr)^{(d-1)^2n\left(1+O\left(\frac{\mathrm{log}(d-1)}{\mathrm{log}\,n}\right)\right)}\leq{t(n,d)}\leq\biggl(\frac{n}{d-1}\biggr)^{(d-1)^2n\left(1+O\left(\frac{\mathrm{log}\,\mathrm{log}(n/(d-1))}{d\,\mathrm{log}(n/(d-1))}\right)\right)}
$$
when $n/d$ goes to infinity.
\end{thm}

Combining Proposition \ref{BP.sec.1.prop.2} and Theorem \ref{BP.sec.1.thm.1} makes it possible to provide estimates on the number of combinatorial types of zonotopes in terms of their dimension and number of generators. We prove Theorem \ref{BP.sec.1.thm.2} when $d$ is fixed and $n$ grows large instead of when $n/d$ grows large as in the statement of Theorem~\ref{BP.sec.1.thm.1} because we will only need it in the $3$-dimensional case. 

\begin{proof}[Proof of Theorem \ref{BP.sec.1.thm.2}]
When $n$ and $d$ are both greater than $1$,
$$
\frac{n}{d-1}=n^{1-\frac{\mathrm{log}(d-1)}{\mathrm{log}\,n}}\mbox{.}
$$

Hence, it follows from Theorem \ref{BP.sec.1.thm.1} that
$$
n^{(d-1)^2n\left(1+O\left(\frac{1}{\mathrm{log}\,n}\right)\right)}\leq{t(n,d)}\leq{n^{(d-1)^2n\left(1+O\left(\frac{\mathrm{log}\,\mathrm{log}\,n}{\mathrm{log}\,n}\right)\right)}}
$$
when $d$ is fixed and $n$ goes to infinity. Hence, it suffices to show that
$$
\frac{t(n,d)}{2^nn!}\leq{z(n,d)}\leq{t(n,d)}\mbox{,}
$$
and use Stirling's approximation formula which implies
$$
2^nn!=n^{n\left(1+O\left(\frac{1}{\mathrm{log}\,n}\right)\right)}
$$
when $n$ goes to infinity.

Let $Z$ be a $d$-dimensional zonotope with $n$ generators contained in $\mathbb{R}^d$. Note that $Z$ admits sets of generators that are contained in an open half-space of $\mathbb{R}^d$ bounded by a hyperplane through the origin: these sets of generators can be obtained by appropriately negating some of the vectors in an arbitrary set of generators of $Z$. Denote by $\mathcal{M}(Z)$ the union of the equivalence classes modulo oriented matroid equivalence of all these sets of generators of $Z$. Further note that any oriented matroid in any of these equivalence classes is the oriented matroid of a set of $n$ vectors spanning $\mathbb{R}^d$ and whose last coordinate is positive. This is due to the fact that isometries of $\mathbb{R}^d$ do not change the oriented matroid of a set of vectors. It follows from Proposition \ref{BP.sec.1.prop.2} that a zonotope $Z'$ has the same combinatorial type as $Z$ if and only if $\mathcal{M}(Z')$ coincides with $\mathcal{M}(Z)$. This shows, in particular that $z(n,d)$ is at most $t(n,d)$.

By construction, the oriented matroid of a set $\mathcal{G}$ of $n$ vectors that span $\mathbb{R}^d$ and whose last coordinate is positive is contained in $\mathcal{M}(Z)$, where $Z$ is any zonotope that admits $\mathcal{G}$ as a set of generators. Moreover, $\mathcal{M}(Z)$ is the union of at most $2^nn!$ equivalence classes modulo oriented matroid equivalence because $Z$ has $2^n$ sets of generators and each of these sets has at most $n!$ oriented matroids. Therefore, $t(n,d)$ is at most $2^nn!$ times $z(n,d)$, as desired.
\end{proof}

When $k$ is even, one obtains from Proposition \ref{BP.sec.1.prop.1} and Theorem \ref{BP.sec.1.thm.2} that the number of $k$-equiprojective polytopes is at least
$$
\biggl(\frac{k}{2}\biggr)^{k\left(\frac{3}{2}+O\left(\frac{1}{\mathrm{log}\,k}\right)\right)}\mbox{.}
$$

The special case of Theorem \ref{BP.sec.0.thm.1} when $k$ is even follows from this inequality and the observation that, if $k$ is greater than $1$, then
$$
\frac{k}{2}=k^{1-\frac{\mathrm{log}\,2}{\mathrm{log}\,k}}\mbox{.}
$$

\section{Aggregated cones}\label{BP.sec.2}

In order to prove Theorem \ref{BP.sec.0.thm.1} when $k$ is odd, we will use a construction of equiprojective polytopes via Minkowski sums. The behavior of equiprojectivity with respect to Minkowski sums will be determined via Theorem \ref{BP.sec.0.thm.2}, whose proof is given in this section. Our starting point for this proof is the article by Masud Hasan and Anna Lubiw \cite{HasanLubiw2008}. Let us recall some of the terminology introduced in that article. In the following definition, given at the beginning of Section~2 in~\cite{HasanLubiw2008}, an \emph{edge-facet incidence} of a $3$-dimensional polytope $P$ is a pair $(e,F)$ such that $F$ is a facet of $P$ and $e$ an edge of $F$.

\begin{defn}
Two edge-facet incidences $(e,F)$ and $(e',F')$ of a $3$-dimensional polytope $P$ \emph{compensate} when $e$ and $e'$ are parallel, and either
\begin{enumerate}
\item[(i)] $F$ and $F'$ coincide but $e$ and $e'$ are distinct or
\item[(ii)] $F$ and $F'$ are parallel and distinct facets of $P$ whose relative interiors are on the same side of the plane that contains $e$ and $e'$. 
\end{enumerate}
\end{defn}

The following theorem is Theorem 1 from \cite{HasanLubiw2008}.

\begin{thm}\label{BP.sec.4.thm.1}
A $3$-dimensional polytope is equiprojective if and only if the set of its edge-facet incidences can be partitioned into compensating pairs.
\end{thm}

Now recall that, given a polytope $P$ contained in $\mathbb{R}^d$ (of dimension possibly less than $d$) and a face $F$ of $P$, the \emph{normal cone of $P$ at $F$} is defined as
$$
N_P(F)=\bigl\{u\in\mathbb{R}^d:\forall\,(x,y)\in{P\mathord{\times}F},\,u\mathord{\cdot}x\leq{u\mathord{\cdot}y}\bigr\}\mbox{.}
$$

The normal cone of a $j$-dimensional face of $P$ is a $(d-j)$-dimensional closed polyhedral cone. In particular, if $P$ is a polytope of any dimension contained in $\mathbb{R}^3$, then a normal cone of $P$ is $2$-dimensional if and only if it is a normal one of $P$ at one of its edges. Moreover, the normal cones of $P$ at two of its edges span the same plane when these edges are parallel. Recall that the plane through the origin of $\mathbb{R}^3$ orthogonal to a non-zero vector $u$ is denoted by $u^\perp$. 

\begin{defn}\label{BP.sec.4.defn.1}
Consider a polytope $P$ and a non-zero vector $u$, both contained in $\mathbb{R}^3$. The \emph{aggregated cone} $C_P(u)$ of $P$ at $u$ is the union of all the $2$-dimensional normal cones of $P$ that are contained in $u^\perp$.
\end{defn}

It should be noted that, while an aggregated cone of $P$ is not necessarily convex, it is still a cone in the sense that it is invariant upon multiplication by a positive number. Further observe that in the statement of Definition \ref{BP.sec.4.defn.1}, if there is no edge of $P$ parallel to $u$, then $C_P(u)$ is empty. For this reason, we are only really interested in the vectors that model the edge directions of a polytope. In addition, it will be useful in the sequel to have just one vector for each possible edge direction and we propose the following definition.

\begin{defn}\label{BP.sec.4.defn.2}
Consider a polytope $P$ contained in $\mathbb{R}^3$ (of dimension possibly less than $3$). A vector $u$ from $\mathbb{S}^2$ is an \emph{edge direction} of $P$ when
\begin{itemize}
\item[(i)] the first non-zero coordinate of $u$ is positive and
\item[(ii)] there is an edge of $P$ parallel to $u$.
\end{itemize}
\end{defn}

Note that edge directions could alternatively be defined as a finite subset of the real projective plane $\mathbb{RP}^2$. The aggregated cones of a polygon or a line segment at their edge directions are particularly well-behaved.

\begin{rem}\label{BP.sec.4.rem.1}
The aggregated cone of a polygon $P$ at an edge direction $u$ is a plane when $P$ has two edges parallel to $u$ and a half-plane when $P$ has a single edge parallel to $u$. Moreover, a line segment has a unique edge direction and its aggregated cone at this edge direction is a plane.
\end{rem}

The notion of an aggregated cone at an edge direction is illustrated in Figures~\ref{BP.sec.4.fig.1} and \ref{BP.sec.4.fig.1.5}. Figure~\ref{BP.sec.4.fig.1.5} shows a triangular prism $P$ and two of its aggregated cones. At the top of the figure, $u$ is the edge direction of $P$ parallel to the three edges $e$, $e'$, and $e''$ of $P$ that are not contained in a triangular face. The orthogonal projection of $P$ along $u$ is depicted at the center of the figure. In that case, $C_P(u)$ coincides with $u^\perp$ as shown on the right of the figure. At the bottom of Figure~\ref{BP.sec.4.fig.1.5}, $u$ is an edge direction parallel to two edges $e$ and $e'$, each contained in one of the triangular faces of $P$ and $C_P(u)$ is the half-plane depicted on the right of the figure. Again, the orthogonal projection of $P$ along $u$ is shown at the center of the figure. Note that $P$ has exactly three edge directions parallel to the edges of its triangular faces. In particular, up to symmetry, Figure~\ref{BP.sec.4.fig.1.5} shows all the aggregated cones of a triangular prism.
\begin{figure}
\begin{centering}
\includegraphics[scale=1]{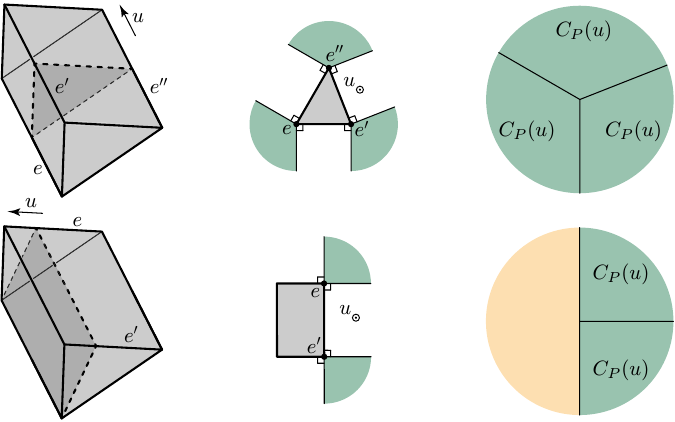}
\caption{A triangular prism (left), its orthogonal projection along $u$ (center), and its aggregated cone $C_P(u)$ at $u$ (right) for two edge directions $u$ (top and bottom).}\label{BP.sec.4.fig.1.5}
\end{centering}
\end{figure}

A regular dodecahedron $P$ and two of its opposite edges $e$ and $e'$ are depicted on the left of Figure \ref{BP.sec.4.fig.1}. Both of these edges are parallel to the same edge direction $u$. The orthogonal projection of $P$ on $u^\perp$ is shown at the center of the figure and the aggregated cone $C_P(u)$ is depicted on the right. It can be seen that $C_P(u)$ is the union of two opposite cones that do not entirely cover $u^\perp$. In particular, this illustrates that, while $C_P(u)$ is always a cone, this cone is not always convex. It should be noted that $P$ is not equiprojective: its orthogonal projection on a plane parallel to a facet is a decagon while its orthogonal projection along a line through two opposite vertices is a dodecagon. However, an example of an equiprojective polytope with a non-convex aggregated cone---the equitruncated tetrahedron \cite{HasanHossainLopez-OrtizNusratQuaderRahman2022}---will be discussed in Section~\ref{BP.sec.5}. Observe that, by the symmetries of the regular dodecahedron, the aggregated cones at the edges of this polytope are all equal up to isometry. In particular, Figure \ref{BP.sec.4.fig.1} depicts all the aggregated cones of the regular dodecahedron.

Note that a facet $F$ of $P$ has at most two edges parallel to any given edge direction $u$ of $P$. Moreover, $F$ has exactly one edge parallel to $u$ if and only if the normal cone of $P$ at $F$ is one of the half-lines that compose the relative boundary of $C_P(u)$. We will use this in the proof of Theorem~\ref{BP.sec.0.thm.2}. This theorem is an equivalence and we shall prove two separate implications.
\begin{figure}
\begin{centering}
\includegraphics[scale=1]{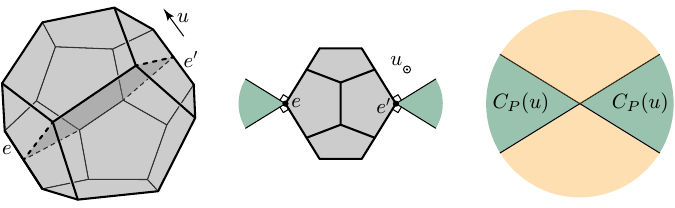}
\caption{A regular dodecahedron $P$ (left), its orthogonal projection along the edge direction $u$ parallel to the edges $e$ and $e'$ (center), and its aggregated cone $C_P(u)$ at $u$ (right).}\label{BP.sec.4.fig.1}
\end{centering}
\end{figure}

\begin{lem}\label{BP.sec.4.lem.1}
Consider a $3$-dimensional polytope $P$. If $P$ is equiprojective, then for any edge direction $u$ of $P$, either
\begin{itemize}
\item[(i)] the aggregated cone $C_P(u)$ is equal to $u^\perp$ or
\item[(ii)] the relative interior of $-C_P(u)$ is equal to $u^\perp\mathord{\setminus}C_P(u)$.
\end{itemize}
\end{lem}

\begin{proof}
Assume that $P$ is equiprojective and consider an edge direction $u$ of $P$ such that $C_P(u)$ is not equal to $u^\perp$. Let $v$ be a unit vector in the relative boundary of $C_P(u)$ such that the relative interior of $C_P(u)$ lies counter-clockwise from $v$. Denote by $w$ the first unit vector in $u^\perp$ counter-clockwise from $v$ that is contained in the relative boundary of $C_P(u)$. Let us show in a first step that $-v$ and $-w$ both belong to the relative boundary of $C_P(u)$.
\begin{figure}
\begin{centering}
\includegraphics[scale=1]{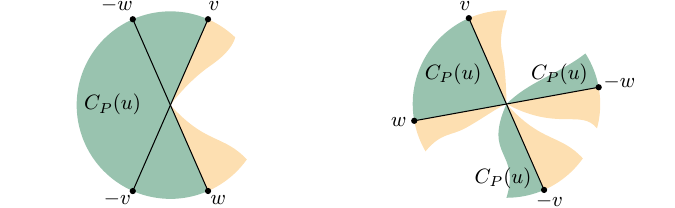}
\caption{An illustration of the proof of Lemma \ref{BP.sec.4.lem.1}.}\label{BP.sec.4.fig.2}
\end{centering}
\end{figure}

As discussed above, $v$ and $w$ span the normal cones of $P$ at two facets $F$ and $G$, respectively, that each admit a unique edge parallel to $u$. Denote these edges of $F$ and $G$ by $e_F$ and $e_G$. Since $P$ is equiprojective, it follows from Theorem~\ref{BP.sec.4.thm.1} that the edge-facet incidence $(e_F,F)$ must be compensated, but since $F$ doesn't have another edge parallel to $e_F$, there must exist a facet $F'$ of $P$ distinct from but parallel to $F$ and an edge $e_{F'}$ of $F'$ parallel to $e_F$ such that the relative interiors of $F$ and $F'$ are on the same side of the plane that contains these two edges. Now observe that $F'$ cannot have another edge parallel to $e_{F'}$. Indeed the edge-facet incidence formed by such an edge with $F'$ couldn't be compensated by any other edge-facet incidence than $(e_{F'},F')$ but this one already compensates $(e_F,F)$. Hence $-v$ is in the relative boundary of $C_P(u)$. By symmetry, $-w$ is also in the relative boundary of $C_P(u)$.

Since $-v$ belongs to the relative boundary of $C_P(u)$, it cannot lie before $w$ counterclockwise from $v$ as shown on the left of Figure \ref{BP.sec.4.fig.2}. Hence, $v$, $w$, and all the unit vectors between them counter-clockwise from $v$ are contained in a half-plane as shown on the right of the figure. As the relative interiors of $F$ and $F'$ are on the same side of the plane that contains $e_F$ and $e_{F'}$, the relative interior of $C_P(u)$ must lie clockwise from $-v$ as shown on the right of Figure \ref{BP.sec.4.fig.2}. Likewise, the relative interior of $C_P(u)$ lies counter-clockwise from $-w$. In particular, all the unit vectors in $u^\perp$ counter-clockwise from $-v$ and clockwise from $-w$ must belong to $u^\perp\mathord{\setminus}C_P(u)$. Indeed, otherwise one of these vectors would belong to the relative boundary of $C_P(u)$ and, as we have just established, its opposite would be a unit vector in the relative boundary of $C_P(u)$ that would lie before $w$ counter-clockwise from $v$, contradicting our assumption on $w$.

This shows that the relative interior of $-C_P(u)$ is contained in $u^\perp\mathord{\setminus}C_P(u)$. By symmetry, one can exchange $C_P(u)$ with the closure of $u^\perp\mathord{\setminus}C_P(u)$ in the above argument and the inverse inclusion therefore also holds.
\end{proof}

The following lemma states the other implication from Theorem \ref{BP.sec.0.thm.2}. It will be proven via Theorem \ref{BP.sec.4.thm.1} by constructing an explicit partition of the set of all edge-facet incidences into compensating pairs.

\begin{lem}\label{BP.sec.4.lem.2}
If, for any edge direction $u$ of a $3$-dimensional polytope $P$,
\begin{itemize}
\item[(i)] the aggregated cone $C_P(u)$ is equal to $u^\perp$ or
\item[(ii)] the relative interior of $-C_P(u)$ is equal to $u^\perp\mathord{\setminus}C_P(u)$.
\end{itemize}
then $P$ is equiprojective.
\end{lem}
\begin{proof}
Assume that the condition in the statement of the lemma holds for any edge direction $u$ of a $3$-dimensional polytope $P$. We will partition the edge-facet incidences of $P$ into compensating pairs and the desired result will follow from Theorem \ref{BP.sec.4.thm.1}. Consider a facet $F$ of $P$ and an edge $e$ of $F$. Denote by $u$ the edge direction of $P$ parallel to $e$ and by $v$ the unit vector that spans the normal cone of $P$ at $F$. We consider two mutually-exclusive cases.

First, assume that $e$ is the only edge of $F$ parallel to $u$. In that case, $C_P(u)$ is not equal to $u^\perp$ and $v$ belongs to the relative boundary of $C_P(u)$. According to assertion (ii), so does $-v$. Therefore, $P$ has a facet $F'$ parallel to and distinct from $F$ that has a unique edge $e'$ parallel to $u$. By the same assertion, the relative interior of $C_P(u)$ either lies counter-clockwise from $u$ and clockwise from $-u$ or clockwise from $u$ and counter-clockwise from $-u$. As a consequence, the relative interiors of $F$ and $F'$ lie on the same side of the plane that contains $e$ and $e'$. This shows that $(e',F')$ compensates $(e,F)$ and we assign these two edge-facet incidences to a pair in the announced partition. Observe that the same process starting with $(e',F')$ instead of $(e,F)$ would have resulted in the same pair of compensating edge-facet incidences of $P$.

Now assume that $F$ has an edge $e'$ parallel to $u$ other than $e$. In that case, $(e,F)$ and $(e',F)$ compensate and we assign these two edge-facet incidences to a pair of the announced partition. Again, starting with $(e',F')$ instead of $(e,F)$ results in the same pair of edge-facet incidences of $P$. Hence, repeating this process for all the edge-facet incidences of $P$ allows to form a partition of the edge-facet incidences of $P$ into compensating pairs, as desired.
\end{proof}

Theorem \ref{BP.sec.0.thm.2} is obtained by combining Lemmas \ref{BP.sec.4.lem.1} and \ref{BP.sec.4.lem.2}. Let us illustrate this theorem on our examples from Figures~\ref{BP.sec.4.fig.1.5} and~\ref{BP.sec.4.fig.1}. It is well known that a triangular prism is equiprojective \cite{HasanHossainLopez-OrtizNusratQuaderRahman2022,HasanLubiw2008}. This can be immediately be recovered from Theorem~\ref{BP.sec.0.thm.2} by observing that the aggregated cones of a triangular prism at its edge directions are either planes or half-planes as shown in Figure~\ref{BP.sec.4.fig.1.5}. In contrast, a regular dodecahedron is not equiprojective and this can also be obtained from Theorem \ref{BP.sec.0.thm.2}. Indeed, the aggregated cones of a regular dodecahedron at its edge directions, depicted in Figure~\ref{BP.sec.4.fig.1}, are the union of two opposite cones that do not collectively cover the plane that they span.

\section{Equiprojectivity and Minkowski sums}\label{BP.sec.2.5}

In this section, we exploit Theorem \ref{BP.sec.0.thm.2} in order to study the equiprojectivity of Minkowski sums. We will also determine the value of $k$ for which a polytope or a Minkowski sum of polytopes is $k$-equiprojective.

\begin{defn}\label{BP.sec.4.defn.3}
Consider an edge direction $u$ of a polytope $P$ contained in $\mathbb{R}^3$. The \emph{multiplicity} $\mu_P(u)$ of $u$ as an edge direction of $P$ is equal to $2$ when $C_P(u)$ coincides with $u^\perp$ and to $1$ when it does not.
\end{defn}

Note that, in Definition \ref{BP.sec.4.defn.3}, $P$ can be a $3$-dimensional polytope, a polygon or a line segment. From now on, we denote by $\kappa(P)$ the value of $k$ such that an equiprojective polytope $P$ is $k$-equiprojective.

\begin{thm}\label{BP.sec.4.thm.4}
If $P$ is an equiprojective polytope, then
$$
\kappa(P)=\sum\mu_P(u)
$$
where the sum is over the edge directions $u$ of $P$.
\end{thm}
\begin{proof}
Consider an equiprojective polytope $P$ contained in $\mathbb{R}^3$ and a plane $H$ through the origin of $\mathbb{R}^3$ that is not orthogonal to any of the facets of $P$. Denote by $\pi:\mathbb{R}^3\rightarrow{H}$ the orthogonal projection on $H$. Since $H$ is not orthogonal to any facet of $P$, a line segment is an edge of $\pi(P)$ if and only if it is the image by $\pi$ of an edge $e$ of $P$ such that the relative interior of $N_P(e)$ intersects $H$. Hence, $\kappa(P)$ is the number of the normal cones of $P$ at its edges whose relative interior is non-disjoint from $H$. Let us count these normal cones of $P$.

We consider an edge direction $u$ of $P$ and count the normal cones of $P$ at its edges parallel to $u$, whose relative interior intersects $H$. Since $H$ is not orthogonal to any facet of $P$, the planes $H$ and $u^\perp$ are distinct and their intersection is a straight line $L$. For the same reason, this line cannot contain any of the normal cones of $P$ at a facet. In other words, $L$ intersects exactly two normal cones of $P$ in their relative interior and these normal cones are at a vertex or at an edge of $P$. According to Theorem \ref{BP.sec.0.thm.2}, two cases are possible: either $C_P(u)$ coincides with $u^\perp$ or the relative interior of $-C_P(u)$ is equal to $u^\perp\mathord{\setminus}C_P(u)$. In the former case, $\mu_P(u)$ is equal to $2$ and $L$ intersects the relative interior of two normal cones of $P$ at distinct edges parallel to $u$. In the latter case, $\mu_P(u)$ is equal to $1$ and $L$ intersects the relative interior of the normal cones of $P$ at an edge parallel to $u$ and at a vertex. Hence, summing $\mu_P(u)$ over the edge directions $u$ of $P$ provides the desired equality.
\end{proof}

Now recall that the faces of the Minkowski sum of two polytopes $P$ and $Q$ can be recovered from the normal cones of these polytopes (see for instance Proposition 7.12 in \cite{Ziegler1995}): the faces of $P+Q$ are precisely the Minkowski sums of a face $F$ of $P$ with a face $G$ of $Q$ such that the relative interiors of $N_P(F)$ and $N_Q(G)$ are non-disjoint. For this reason, Theorem \ref{BP.sec.0.thm.2} provides a convenient way to determine how equiprojectivity behaves under Minkowski sums.

\begin{thm}\label{BP.sec.4.thm.5}
Let $P$ and $Q$ each be an equiprojective polytope, a polygon, or a line segment contained in $\mathbb{R}^3$. The Minkowski sum $P+Q$ is equiprojective if and only if for each edge direction $u$ shared by $P$ and $Q$, either
\begin{itemize}
\item[(i)] $C_P(u)$ or $C_Q(u)$ is equal to $u^\perp$ or
\item[(ii)] $C_P(u)$ coincides with $C_Q(u)$ or with $-C_Q(u)$.
\end{itemize}
\end{thm}
\begin{proof}
Observe that according to Proposition 7.12 in \cite{Ziegler1995},
\begin{equation}\label{BP.sec.4.thm.5.eq.1}
C_{P+Q}(u)=C_P(u)\cup{C_Q(u)}
\end{equation}
for every edge direction $u$ of $P+Q$.

First assume that, for each edge direction $u$ of $P+Q$, at least one of the two assertions (i) and (ii) from the statement of the theorem holds. Pick an edge direction $u$ of $P+Q$ such that $C_{P+Q}(u)$ is not equal to $u^\perp$ and let us prove that the relative interior of $-C_{P+Q}(u)$ is equal to $u^\perp\mathord{\setminus}C_{P+Q}(u)$. It will then follow from Theorem~\ref{BP.sec.0.thm.2} that $P+Q$ is an equiprojective polytope.

Note that $u$ must be an edge direction of $P$ or one of $Q$. If $u$ is not an edge direction of both then it follows from (\ref{BP.sec.4.thm.5.eq.1}) that the aggregated cone $C_{P+Q}(u)$ is equal to $C_P(u)$ or to $C_Q(u)$. In that case, Theorem~\ref{BP.sec.0.thm.2} and Remark~\ref{BP.sec.4.rem.1} imply that the relative interior of $-C_{P+Q}(u)$ coincides with $u^\perp\mathord{\setminus}C_{P+Q}(u)$, as desired. If $u$ is an edge direction of both $P$ and $Q$, then by our assumption, either one of the cones $C_P(u)$ and $C_Q(u)$ is equal to $u^\perp$ or the cone $C_P(u)$ coincides with $C_Q(u)$ or with $-C_Q(u)$. However, $C_{P+Q}(u)$ is not equal to $u^\perp$. Therefore, because of (\ref{BP.sec.4.thm.5.eq.1}), this implies that $C_P(u)$ coincides with $C_Q(u)$. In that case, again by (\ref{BP.sec.4.thm.5.eq.1}), $C_{P+Q}(u)$ is equal to $C_P(u)$ and in turn, according to Theorem~\ref{BP.sec.0.thm.2} and Remark~\ref{BP.sec.4.rem.1}, the relative interior of $-C_P(u)$ is equal to $u^\perp\mathord{\setminus}C_P(u)$. It follows that $-C_{P-Q}(u)$ is equal to $u^\perp\mathord{\setminus}C_{P+Q}(u)$, as desired.

Now assume that $P+Q$ is equiprojective and that $u$ is an edge direction shared by $P$ and $Q$ such that neither $C_P(u)$ or $C_Q(u)$ is equal to $u^\perp$. Observe that $C_P(u)\cap\mathbb{S}^2$ is a union of finitely many circular arcs and denote by $\alpha$ the sum of the lengths of these arcs. Likewise, denote by $\beta$ and $\gamma$ the sum of the lengths of the circular arcs contained in $C_P(u)\cap\mathbb{S}^2$ and in $C_{P+Q}(u)$, respectively. Since $P$ and $Q$ each are an equiprojective polytope, a polygon, or a line segment, whose aggregated cones at $u$ are not equal to $u^\perp$, it follows from Theorem~\ref{BP.sec.0.thm.2} and from Remark~\ref{BP.sec.4.rem.1} that the relative interiors of $-C_P(u)$ and $-C_Q(u)$ coincide with $u^\perp\mathord{\setminus}C_P(u)$ and $u^\perp\mathord{\setminus}C_Q(u)$, respectively. Hence, $\alpha$ and $\beta$ are both equal to half the circumference of a unit circle. Also according to Theorem~\ref{BP.sec.0.thm.2}, either $C_{P+Q}(u)$ is equal to $u^\perp$ or the relative interior of $-C_{P+Q}(u)$ coincides with $u^\perp\mathord{\setminus}C_{P+Q}(u)$. Therefore, $\gamma$ is equal to the circumference of a unit circle or to half of it. In the former case, $\gamma$ is equal to the sum of $\alpha$ and $\beta$. By (\ref{BP.sec.4.thm.5.eq.1}), this can only happen when $C_P(u)$ coincides with $-C_Q(u)$. In the latter case, $\alpha$, $\beta$, and $\gamma$ are equal and it follows from (\ref{BP.sec.4.thm.5.eq.1}) that $C_P(u)$ and $C_Q(u)$ coincide.
\end{proof}

Recall that the multiplicity of edge directions is defined for polygons and line segments and not just for $3$\nobreakdash-dimensional polytopes. From now on, if $P$ is a polygon or a line segment contained in $\mathbb{R}^3$, we denote
$$
\kappa(P)=\sum\mu_P(u)
$$
by analogy with Theorem \ref{BP.sec.4.thm.4}, where the sum is over the edge directions $u$ of $P$. It should be noted that when $P$ is a polygon, $\kappa(P)$ is equal to the number of edges of $P$ and when $P$ is a line segment, $\kappa(P)$ is always equal to $2$.

Consider two polytopes $P$ and $Q$ contained in $\mathbb{R}^3$. If each of them is an equiprojective polytope, a polygon, or a line segment, then we denote
$$
\lambda(P,Q)=r+2s
$$
where $s$ is the number of edge directions $u$ common to $P$ and $Q$ such that both $C_P(u)$ and $C_Q(u)$ are equal to $u^\perp$ while $r$ is the number of edge directions $u$ common to $P$ and $Q$ such that the cones $C_P(u)$ and $C_Q(u)$ are not both equal to $u^\perp$ and one of them is contained in the other.

\begin{thm}\label{BP.sec.4.thm.6}
Consider two polytopes $P$ and $Q$ contained in $\mathbb{R}^3$, each of which is an equiprojective polytope, a polygon, or a line segment. If the Minkowski sum $P+Q$ is an equiprojective polytope, then
$$
\kappa(P+Q)=\kappa(P)+\kappa(Q)-\lambda(P,Q)\mbox{.}
$$
\end{thm}
\begin{proof}
Assume that $P+Q$ is equiprojective and consider an edge direction $u$ of $P+Q$. Note that $u$ is then an edge direction of $P$ or an edge direction of $Q$. 
It follows from Proposition 7.12 in \cite{Ziegler1995} that
$$
C_{P+Q}(u)=C_P(u)\cup{C_Q(u)}\mbox{.}
$$

Hence, if $u$ is an edge direction of $P$ but not one of $Q$, then
\begin{equation}\label{BP.sec.4.thm.6.eq.1}
\mu_{P+Q}(u)=\mu_P(u)
\end{equation}
and if $u$ is an edge direction of $Q$ but not $P$, then
\begin{equation}\label{BP.sec.4.thm.6.eq.2}
\mu_{P+Q}(u)=\mu_Q(u)\mbox{.}
\end{equation}

If $P$ and $Q$ share $u$ as an edge direction, we review the different possibilities given by Theorem \ref{BP.sec.4.thm.5} for how $\mu_{P+Q}(u)$, $\mu_P(u)$ and $\mu_Q(u)$ relate to one another. If both the aggregated cones $C_P(u)$ and $C_Q(u)$ are equal to $u^\perp$, then $\mu_P(u)$, $\mu_Q(u)$ and $\mu_{P+Q}(u)$ are all equal to $2$. Hence,
\begin{equation}\label{BP.sec.4.thm.6.eq.3}
\mu_{P+Q}(u)=\mu_P(u)+\mu_Q(u)-2\mbox{.}
\end{equation}

If $C_P(u)$ and $C_Q(u)$ are not both equal to $u^\perp$, but one of these aggregated cones is contained in the other, then
$$
\mu_{P+Q}(u)=\mathrm{max}\{\mu_P(u),\mu_Q(u)\}\mbox{.}
$$

Moreover, the smallest value between $\mu_P(u)$ and $\mu_Q(u)$ is equal to $1$. Hence,
\begin{equation}\label{BP.sec.4.thm.6.eq.4}
\mu_{P+Q}(u)=\mu_P(u)+\mu_Q(u)-1\mbox{.}
\end{equation}

Finally, if $C_P(u)$ and $C_Q(u)$ are opposite and not equal to $u^\perp$, then $\mu_P(u)$ and $\mu_Q(u)$ are both equal to $1$ while $\mu_{P+Q}(u)$ is equal to $2$. Therefore,
\begin{equation}\label{BP.sec.4.thm.6.eq.5}
\mu_{P+Q}(u)=\mu_P(u)+\mu_Q(u)\mbox{.}
\end{equation}

The result is then obtained from Theorem \ref{BP.sec.4.thm.4} by summing (\ref{BP.sec.4.thm.6.eq.1}), (\ref{BP.sec.4.thm.6.eq.2}), (\ref{BP.sec.4.thm.6.eq.3}), (\ref{BP.sec.4.thm.6.eq.4}), and (\ref{BP.sec.4.thm.6.eq.5}) when $u$ ranges over the corresponding edge directions of $P+Q$.
\end{proof}

Observe that when two polytopes $P$ and $Q$ do not share an edge direction, $\lambda(P,Q)$ vanishes. Hence, we obtain the following corollary from Theorems \ref{BP.sec.4.thm.5} and \ref{BP.sec.4.thm.6}. In turn, that corollary immediately implies Theorem \ref{BP.sec.0.thm.3}. 

\begin{cor}\label{BP.sec.4.cor.1}
Consider two polytopes $P$ and $Q$ contained in $\mathbb{R}^3$, each of which is an equiprojective polytope, a polygon, or a line segment. If $P$ and $Q$ do not share an edge direction, then $P+Q$ is equiprojective and
$$
\kappa(P+Q)=\kappa(P)+\kappa(Q)\mbox{.}
$$
\end{cor}

\section{Many $k$-equiprojective polytopes when $k$ is odd}\label{BP.sec.3}

The goal of this section is to build many $k$-equiprojective polytopes when $k$ is a large enough odd integer. This will be done using Minkowski sums between zonotopes with $(k-3)/2$ generators and well-chosen triangles. In this section, all the considered polytopes are contained in $\mathbb{R}^3$.

For any $3$-dimensional zonotope $Z$, we consider a triangle $t_Z$ whose edge directions do not belong to any of the planes spanned by two edge directions of $Z$ and such that the plane spanned by the edge directions of $t_Z$ does not contain any edge direction of $Z$. Note that such a triangle always exists because $Z$ and $t_Z$ have only finitely many edge directions. A consequence of these requirements is that $t_Z$ does not share any edge direction with $Z$. Moreover, the normal cones of $Z$ at its facets are never contained in the normal cone of $t_Z$ at itself or in the plane spanned by the normal cone of $t_Z$ at one of its edges (see Figure \ref{BP.sec.3.fig.1} for an illustration of the normal cones of a triangle at its edges and at itself). Inversely, the normal cone of $t_Z$ at itself is not contained in the plane spanned by the normal cone of $Z$ at any of its edges. It should be noted that there are many possible choices for $t_Z$ but we fix that choice for each zonotope $Z$ so that $Z\mapsto{t_Z}$ defines a map that sends a zonotope to a triangle.
\begin{figure}[b]
\begin{centering}
\includegraphics[scale=1]{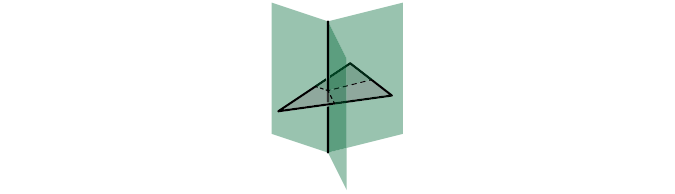}
\caption{A triangle in $\mathbb{R}^3$ (colored gray) and its normal cones at edges (colored green) and at itself (black vertical line).}\label{BP.sec.3.fig.1}
\end{centering}
\end{figure}

It follows from the results of Section \ref{BP.sec.2.5} that the Minkowski sum of $Z$ with $t_Z$ is always an equiprojective polytope.

\begin{lem}\label{BP.sec.3.lem.0.5}
If $Z$ is a $3$-dimensional zonotope with $n$ generators, then its Minkowski sum with $t_Z$ is a $(2n+3)$\nobreakdash-equiprojective polytope.
\end{lem}
\begin{proof}
Recall that the edge directions of $Z+t_Z$ are precisely the edge directions of $Z$ and the edge directions of $t_Z$. 
By Proposition \ref{BP.sec.1.prop.1}, $Z$ is $2n$-equiprojective where $n$ is the number of generators of $Z$.

Now recall that when $P$ is a polygon, $\kappa(P)$ is equal to the number of edges of $P$. As a consequence, $\kappa(t_Z)$ is equal to $3$ and since $Z$ does not share an edge direction with $t_Z$, it follows from Corollary~\ref{BP.sec.4.cor.1} that the Minkowski sum of $Z$ with $t_Z$ is a $(2n+3)$-equiprojective polytope.
\end{proof}

Recall that the normal cones of a polytope $P$ ordered by reverse inclusion collectively form a lattice $\mathcal{N}(P)$, called the \emph{normal fan of $P$} and that
$$
N_P:\mathcal{F}(P)\rightarrow\mathcal{N}(P)
$$
is an isomorphism (see for instance Exercise 7.1 in \cite{Ziegler1995}). Note that in this context, an \emph{isomorphism} is a bijective morphism between two lattices. It will be important to keep in mind that $N_P$ reverses inclusion. This correspondence between the face lattice and the normal fan of a polytope will be useful in order to determine the combinatorial type of the Minkowski sum between a $3$-dimensional zonotope $Z$ and its associated triangle $t_Z$.

Further recall that a face $F$ of $Z+t_Z$ is uniquely written as the Minkowski sum of a face of $Z$ with a face of $t_Z$. We will denote by $\tau_Z(F)$ the face of $t_Z$ that appears in this Minkowski sum. Note that $\tau_Z$ defines a morphism from the face lattice of $Z+t_Z$ to that of $t_Z$ in the sense that it preserves face inclusion. This is a consequence, for instance of Proposition~7.12 from \cite{Ziegler1995}.

The remainder of the section is mostly devoted to proving that if $Z+t_Z$ and $Z'+t_{Z'}$ have the same combinatorial type, then so do $Z$ and $Z'$. This will be done via two lemmas. The first of these lemmas is the following.
\begin{lem}\label{BP.sec.3.lem.1}
Consider two $3$-dimensional zonotopes $Z$ and $Z'$. If
$$
\psi:\mathcal{F}(Z+t_Z)\rightarrow\mathcal{F}(Z'+t_{Z'})
$$
is an isomorphism, then there is a commutative diagram

$$
\begin{tikzcd}
\mathcal{F}(Z+t_Z) \arrow[r,"\psi"] \arrow[d,"\tau_Z"'] &
\mathcal{F}(Z'+t_{Z'}) \arrow[d,"\tau_{Z'}"']
\\
\mathcal{F}(t_Z) \arrow[r,"\phi"] &
\mathcal{F}(t_{Z'})
\end{tikzcd}
$$
where $\phi$ is an isomorphism. 
\end{lem}
\begin{proof}
Assume that $\psi$ is an isomorphism from the face lattice of $Z+t_Z$ to that of $Z'+t_{Z'}$. Observe that $Z+t_Z$ has exactly two parallel triangular facets each of which is a translate of $t_Z$. Further observe that all the other facets of $Z+t_Z$ are centrally-symmetric, since they are the Minkowski sum of a centrally-symmetric polygon with a line segment or a point. The same observations hold for the Minkowski sum of $Z'$ with $t_{Z'}$: that Minkowski sum has exactly two parallel triangular facets, each a translate of  $t_{Z'}$ and its other facets all are centrally-symmetric. As $\psi$ induces an isomorphism between the face lattices of any facet $F$ of $Z+t_Z$ and the face lattice of $\psi(F)$, this shows that $\psi$ sends the two triangular facets of $Z+t_Z$ to the two triangular facets of $Z'+t_{Z'}$. Moreover, two parallel edges of a centrally-symmetric facet of $Z+t_Z$ are sent by $\psi$ to two parallel edges of a centrally-symmetric facet of $Z'+t_{Z'}$.

Recall that all the aggregated cones of $t_Z$ are half-planes. As $Z$ and $t_Z$ do not have any edge direction in common, the aggregated cones of $Z+t_Z$ at the edge directions of $t_Z$ are still half-planes. Note that the two half-lines that bound these half-planes are precisely the normal cones of $Z+t_Z$ at its triangular facets. Similarly, the aggregated cones of $Z'+t_{Z'}$ at the edge directions of $t_{Z'}$ are half-planes bounded by the normal cones of $Z'+t_{Z'}$ at its triangular facets. Since two parallel edges of a centrally-symmetric facet of $Z+t_Z$ are sent by $\psi$ to two parallel edges of a centrally-symmetric facet of $Z'+t_{Z'}$, this implies that for every edge $e$ of $t_Z$, there exists an edge $\phi(e)$ of $t_{Z'}$ such that any face of $Z+t_Z$ contained in $\tau_Z^{-1}(\{e\})$ is sent to $\phi(e)$ by $\tau_{Z'}\circ\psi$.

By the correspondence between face lattices and normal fans,
$$
\overline{\psi}=N_{Z'+t_{Z'}}\circ\psi\circ{N_{Z+t_Z}^{-1}}
$$
provides an isomorphism from $\mathcal{N}(Z+t_Z)$ to $\mathcal{N}(Z'+t_{Z'})$. Consider two edges $e$ and $f$ of $t_Z$ and denote by $x$ the vertex they share. Let $\mathcal{P}$ be the set of the normal cones of $Z+t_Z$ at its two triangular faces and at the faces contained in $\tau_Z^{-1}(\{e\})$ and in $\tau_Z^{-1}(\{f\})$. Similarly, let $\mathcal{P}'$ denote the set made up of the normal cones of $Z'+t_{Z'}$ at its triangular faces and at the faces from $\tau_{Z'}^{-1}(\{\phi(e)\})$ and $\tau_{Z'}^{-1}(\{\phi(f)\})$. By the above, $\overline{\psi}(\mathcal{P})$ is equal to $\mathcal{P}'$. Moreover, it follows from Proposition~7.12 in \cite{Ziegler1995} that $\mathcal{P}$ and $\mathcal{P}'$ are polyhedral decompositions of the boundaries of the normal cone of $t_Z$ at $x$ and of the normal cone of $t_{Z'}$ at the vertex $\phi(x)$ shared by $\phi(e)$ and $\phi(f)$. Observe that the boundaries of these normal cones (that are depicted in Figure \ref{BP.sec.3.fig.1}) each separate $\mathbb{R}^3$ into two connected components. As $\overline{\psi}$ is an isomorphism from $\mathcal{N}(Z+t_Z)$ to $\mathcal{N}(Z'+t_{Z'})$, it must then send the normal cones of $Z+t_Z$ at the faces from $\tau_Z^{-1}(\{x\})$ to the normal cones of $Z'+t_{Z'}$ at the faces from $\tau_{Z'}^{-1}(\{\phi(x)\})$. In other words, any face of $Z+t_Z$ contained in $\tau_Z^{-1}(\{x\})$ is sent to $\phi(x)$ by $\tau_{Z'}\circ\psi$. Hence, setting $\phi(t_Z)$ to $t_{Z'}$ results in the desired isomorphism $\phi$.
\end{proof}

Recall that, for an arbitrary $3$-dimensional zonotope $Z$, a face $F$ of $Z+t_Z$ is the Minkowski sum of a unique face of $Z$ with $\tau_Z(F)$. We will denote by $\zeta_Z(F)$ the face of $Z$ that appears in this sum. This defines a morphism $\zeta_Z$ from $\mathcal{F}(Z+t_Z)$ to $\mathcal{F}(Z)$. We will derive a statement similar to Lemma \ref{BP.sec.3.lem.1} regarding $\zeta_Z$. In order to do that, we first establish the following proposition as a consequence of our requirements for the choice of $t_Z$.

\begin{prop}\label{BP.sec.3.prop.1.5}
Consider a $3$-dimensional zonotope $Z$ and two proper faces $F$ and $G$ of $Z+t_Z$. If $F$ is a facet of $G$, then either
\begin{enumerate}
\item[(i)] $\zeta_Z(G)$ coincides with $\zeta_Z(F)$ or
\item[(ii)] $\tau_Z(G)$ coincides with $\tau_Z(F)$.
\end{enumerate}
\end{prop}
\begin{proof}
Let us first assume that $G$ is a facet of $Z+t_Z$ and $F$ an edge of $G$. It follows from our choice for $t_Z$ that every polygonal face of $Z+t_Z$ is the Minkowski sum of a facet of $Z$ with a vertex of $t_Z$, of a vertex of $Z$ with $t_Z$ itself, or of an edge of $Z$ with an edge of $t_Z$. If $\zeta_Z(G)$ is a polygon and $\tau_Z(G)$ a vertex of $t_Z$ then it is immediate that $F$ is the Minkowski sum of an edge of $\zeta_Z(G)$ with $\tau_Z(G)$ and therefore, $\tau_Z(F)$ coincides with $\tau_Z(G)$. Similarly, if $\zeta_Z(G)$ is a vertex of $Z$ and $\tau_Z(G)$ is equal to $t_Z$, then $\tau_Z(F)$ must be an edge of $t_Z$ and $\zeta_Z(F)$ is equal to $\zeta_Z(G)$. Now if $\zeta_Z(G)$ is an edge of $Z$ and $\tau_Z(G)$ an edge of $t_Z$, then observe that $G$ is a parallelogram. As $F$ is an edge of $G$, it must be a translate of either $\zeta_Z(G)$ or $\tau_Z(G)$. In the former case, $\zeta_Z(F)$ is equal to $\zeta_Z(G)$ and in the latter, $\tau_Z(F)$ is equal to $\tau_Z(G)$, as desired.

Finally, assume that $G$ is an edge of $Z+t_Z$ and $F$ a vertex of $G$. Recall that $Z$ and $t_Z$ do not share an edge direction. As a consequence, either $\zeta_Z(G)$ is an edge of $Z$ and $\tau_Z(G)$ a vertex of $t_Z$ or inversely, $\zeta_Z(G)$ is a vertex of $Z$ and $\tau_Z(G)$ an edge of $t_Z$. In the former case, $\tau_Z(F)$ is equal to $\tau_Z(G)$ and in the latter $\zeta_Z(F)$ coincides with $\zeta_Z(G)$, which completes the proof.
\end{proof}

We can now prove the following statement similar to that of Lemma \ref{BP.sec.3.lem.1}.

\begin{lem}\label{BP.sec.3.lem.1.5}
Consider two $3$-dimensional zonotopes $Z$ and $Z'$. If
$$
\psi:\mathcal{F}(Z+t_Z)\rightarrow\mathcal{F}(Z'+t_{Z'})
$$
is an isomorphism, then there is a commutative diagram
$$
\begin{tikzcd}
\mathcal{F}(Z+t_Z) \arrow[r,"\psi"] \arrow[d,"\zeta_Z"'] &
\mathcal{F}(Z'+t_{Z'}) \arrow[d,"\zeta_{Z'}"']
\\
\mathcal{F}(Z) \arrow[r,"\theta"] &
\mathcal{F}(Z')
\end{tikzcd}
$$
where $\theta$ is an isomorphism.
\end{lem}
\begin{proof}
Consider a proper face $F$ of $Z$ and let us show that all the faces of $Z+t_Z$ contained in $\zeta_Z^{-1}(\{F\})$ have the same image by $\zeta_{Z'}\circ\psi$. Assume for contradiction that this is not the case and recall that by Proposition~7.12 from~\cite{Ziegler1995}, the normal cone $N_Z(F)$ is decomposed into a polyhedral complex by the normal cones of $Z+t_Z$ it contains. Since $N_Z(F)$ is convex and
$$
N_{Z'+t_{Z'}}:\mathcal{F}(Z'+t_{Z'})\rightarrow\mathcal{N}(Z'+t_{Z'})
$$
is an isomorphism, $\zeta_Z^{-1}(\{F\})$ must contain two faces $P$ and $Q$ whose images by $\zeta_{Z'}\circ\psi$ differ and such that the normal cone of $Z+t_Z$ at $Q$ is a facet of the normal cone of $Z+t_Z$ at $P$. By the correspondence between the face lattice of a polytope and its normal fan, $P$ is a facet of $Q$. Now recall that
$$
P=F+\tau_Z(P)
$$
and
$$
Q=F+\tau_Z(Q)\mbox{.}
$$

It follows that $\tau_Z(P)$ must differ from $\tau_Z(Q)$ as $P$ and $Q$ would otherwise be equal. Hence, by Lemma~\ref{BP.sec.3.lem.1}, $\psi(P)$ and $\psi(Q)$ have different images by $\tau_{Z'}$. As they also have different images by $\zeta_{Z'}$ and as $\psi(P)$ is a facet of $\psi(Q)$, this contradicts Proposition \ref{BP.sec.3.prop.1.5}. As a consequence, all the faces of $Z+t_Z$ contained in $\zeta_Z^{-1}(\{F\})$ have the same image by $\psi\circ\zeta_{Z'}$. In other words, there exists a face $\theta(F)$ of $Z'$ such that $\psi$ sends $\zeta_Z^{-1}(\{F\})$ to a subset of $\zeta_{Z'}^{-1}(\{\theta(F)\})$. In fact, this subset is necessarily $\zeta_{Z'}^{-1}(\{\theta(F)\})$ itself. Indeed, observe that 
$$
\Bigl\{\zeta_Z^{-1}(\{F\}):F\in\mathcal{F}(Z)\Bigr\}
$$
is a partition of $\mathcal{F}(Z+t_Z)$. Similarly, 
$$
\Bigl\{\zeta_{Z'}^{-1}(\{F\}):F\in\mathcal{F}(Z')\Bigr\}
$$
is a partition of $\mathcal{F}(Z'+t_{Z'})$. As $\psi$ is a bijection, it must then send $\zeta_Z^{-1}(\{F\})$ precisely to $\zeta_{Z'}^{-1}(\{\theta(F)\})$. It follows that $\theta$ is a bijection from $\mathcal{F}(Z)$ to $\mathcal{F}(Z')$ such that $\theta\circ\zeta_Z$ is equal to $\zeta_{Z'}\circ\psi$, as desired.

Finally, if $G$ is a face of $Z$ and $F$ a face of $G$, then observe that some face of $Z+t_Z$ from $\zeta_Z^{-1}(\{F\})$ must be contained in a face from $\zeta_Z^{-1}(\{G\})$. Since $\zeta_{Z'}\circ\psi$ is a morphism, this shows that $\theta(F)$ is a face of $\theta(G)$. Hence $\theta$ is an isomorphism from the face lattice of $Z$ to that of $Z'$.
\end{proof}

The following is an immediate consequence of Lemma \ref{BP.sec.3.lem.1.5}.

\begin{lem}\label{BP.sec.3.lem.2}
Consider two $3$-dimensional zonotope $Z$ and $Z'$. If the Minkowski sums $Z+t_Z$ and $Z'+t_{Z'}$ have the same combinatorial type, then the zonotopes $Z$ and $Z'$ have the same combinatorial type.
\end{lem}

Consider an odd integer $k$ greater than or equal to $9$ (under this assumption, there exist zonotopes with $(k-3)/2$ generators). It follows from Lemmas \ref{BP.sec.3.lem.0.5} and \ref{BP.sec.3.lem.2} that the number of different combinatorial types of $k$-equiprojective polytopes is at least the number of zonotopes with $(k-3)/2$ generators. Hence, Theorem \ref{BP.sec.0.thm.1} in the case of the $k$-equiprojective polytopes such that $k$ is odd follows from Theorem \ref{BP.sec.1.thm.2} and from the observations that
$$
\frac{k-3}{\mathrm{log}\,\frac{k-3}{2}}=O\!\left(\frac{k}{\mathrm{log}\,k}\right)
$$
and that
$$
\left(\frac{k-3}{2}\right)^{\!\frac{3}{2}(k-3)}=k^{k\left(\frac{3}{2}+O\left(\frac{1}{\mathrm{log}\,k}\right)\right)}
$$
as $k$ goes to infinity.

\section{Equiprojective polytopes and decomposability}\label{BP.sec.5}

Our results show that Minkowski sums allow for a sequential construction of equiprojective polytopes, thus providing a partial answer to Shephard's original question. Two of the possible kinds of summands in these Minkowski sums, line segments and polygons, are well understood. However, equiprojective polytopes can also appear as a summand, and one may ask what the primitive building blocks of these sequential Minkowski sum constructions look like. More precisely, recall that a polytope $P$ is \emph{decomposable} when it can be written as a Minkowski sum of two polytopes, none of which is homothetic to $P$ \cite{Kallay1982,PrzeslawskiYost2008,Shephard1963,Smilansky1987}. This notion appears naturally within the study of the deformation cones of polytopes \cite{Meyer1974} and in particular in the case of the submodular cone \cite{PilaudPadrolPoullot2021}. It also appears in the study of the diameter of polytopes \cite{DezaPournin2019}.

We ask the following question in the spirit of Shephard's.

\begin{qtn}\label{BP.sec.0.qtn.1}
Are there indecomposable equiprojective polytopes?
\end{qtn}

\begin{figure}[b]
\begin{centering}
\includegraphics[scale=1]{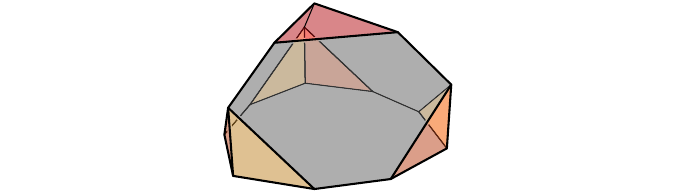}
\caption{The equitruncated tetrahedron from \cite{HasanHossainLopez-OrtizNusratQuaderRahman2022}.}\label{BP.sec.5.fig.1}
\end{centering}
\end{figure}

According to Lemma \ref{BP.sec.4.lem.2}, any $3$-dimensional polytope whose aggregated cones at all edge directions are planes or half-planes is necessarily equiprojective. These equiprojective polytopes form a natural superset of the zonotopes since the aggregated cones of $3$-dimensional zonotopes at their edge directions are planes. It should be noted that this superset of the zonotopes does not contain all equiprojective polytopes. For instance, the equitruncated tetrahedron described in \cite{HasanHossainLopez-OrtizNusratQuaderRahman2022} has an aggregated cone at an edge direction that is neither a plane nor a half-plane. This equiprojective polytope is shown in Figure \ref{BP.sec.5.fig.1}. It is obtained by cutting a tetrahedron $T$ along seven planes. The nonagonal face and the three hexagonal faces colored gray in the figure are what remains of the triangular faces of $T$. The seven triangular faces of the equitruncated tetrahedron each result from one of the cuts. Six of these triangular faces are incident to the nonagonal face and colored yellow and orange in Figure \ref{BP.sec.5.fig.1}. They all have a common edge direction orthogonal to the nonagon. Moreover, each of them shares a single edge direction with $T$. The seventh triangular face, colored red in the figure, is parallel to the nonagonal face. It is not hard to see that the aggregated cone of the equitruncated tetrahedron at the edge direction orthogonal to the nonagonal face is the union of three cones whose pairwise intersection is reduced to the origin of $\mathbb{R}^3$. In particular, this aggregated cone is distinct from both a plane and a half-plane. This is another example of a non-convex aggregated cone. One can also see on the figure that the aggregated cones of the equitruncated tetrahedron at all of its other edge directions are half-planes. Note that this polytope is decomposable and in particular, it admits a tetrahedron homothetic to $T$ as a summand.

Finally, observe that, if the aggregated cone of a $3$-dimensional polytope at some of its edge directions is a plane, then that polytope is decomposable (see for example \cite{DezaPournin2019}). However, the above question on polytope decomposability is open for the $3$-dimensional polytopes whose aggregated cones at all edge directions are half-planes. By the above, these polytopes form a subclass of the equiprojective polytopes disjoint from zonotopes.
 
\begin{qtn}\label{BP.sec.0.qtn.2}
Does there exist an indecomposable $3$-dimensional polytope whose aggregated cones at all edge directions are half-planes?
\end{qtn}

\noindent\textbf{Acknowledgements.} We are grateful to Xavier Goaoc for useful comments on an early version of this article and to anonymous referees for helpful suggestions and remarks. The PhD work of the first author has been partially funded by the MathSTIC (CNRS FR3734) research consortium.
\medskip

\bibliography{EquiprojectivePolytopes}
\bibliographystyle{ijmart}

\end{document}